\documentclass{article}
\usepackage{graphicx} % Required for inserting images
\usepackage{latexsym}
\usepackage{amsfonts, amssymb, mathrsfs, amsmath}%\mathbb
\usepackage{a4wide}
\usepackage{url}
\usepackage{enumerate}
\usepackage{xcolor}
\usepackage{bm}
\usepackage{amsthm}%proof
\usepackage[colorlinks,citecolor=blue,urlcolor=blue]{hyperref}

\newcommand\Prob{{\mathbb P}}
\newcommand\E{{\mathbb E}}
\newcommand\V{{\rm Var}}
\newcommand\cov{{\rm Cov}}

\newtheorem{theorem}{Theorem}

\newtheorem{lemma}{Lemma}

\newtheorem{proposition}{Proposition}

\newcommand\inprob{\buildrel {P} \over  \longrightarrow}
\newcommand\indist{\buildrel {D} \over \longrightarrow }

\DeclareRobustCommand{\stirling}{\genfrac\{\}{0pt}{}}

\allowdisplaybreaks

\title{Limit Laws for the Generalized Zagreb Indices of Random Graphs}
\author{Qunqiang Feng{\thanks{Email: fengqq@ustc.edu.cn}, Hongpeng Ren, Yaru Tian}\\
{\small Department of Statistics and Finance, School of Management} \\
{\small University of Science and Technology of China}\\
{\small Hefei 230026, China}}
\date{}

\begin{document}
\maketitle

\begin{abstract}
In this paper, we study the limiting behavior of the generalized Zagreb indices of 
the classical Erd\H{o}s-R\'{e}nyi (ER) random graph $G(n,p)$, as $n\to\infty$. 
For any integer $k\ge1$, 
we first give an expression for the $k$-th order generalized Zagreb index  
in terms of the number of star graphs of various sizes in any simple graph. 
The explicit formulas for the first two moments of the generalized Zagreb 
indices of an ER random graph are then obtained by this expression.
Based on the asymptotic normality of the numbers of star graphs of various sizes,
%Finally, if the ER random graph is neither too sparse nor close to a complete graph,
several joint limit laws are established for a finite number of
generalized Zagreb indices with a phase transition for $p$ in different regimes.  
Finally, we provide a necessary and sufficient condition for 
any single generalized Zagreb index of $G(n,p)$ to be asymptotic normal. 


\end{abstract}

\section{Introduction}

Topological indices, which are real numbers that represent the topology of (molecular) graphs,
play a significant role in mathematical chemistry, especially in the study of quantitative structure-property 
relationship (QSPR) and quantitative structure-activity relationship (QSAR)
\cite{dev1999top,estrada2022stat}. 
The well-known {\em Zagreb index}, the first degree-based topological index \cite{gutman2013degree}, 
was originally proposed by Gutman and Trinajsti\'{c} more than 
50 years ago  \cite{gutman1972graph}.   
This index, defined as the sum of squares of the degrees of vertices in a graph,
has been utilized to discuss the QSAR/QSPR of the different chemical structures such as chirality, 
complexity, hetero-system, ZE-isomers and $\pi$-electron energy \cite{diudea2001qsar}.
For more details of the mathematical properties and chemical applications of the Zagreb index, 
see, e.g., \cite{gutman2004zagreb,das2015zagreb,wang2016sum,horoldagva2021zagreb} and the references therein.

Let $G=(V,E)$ be a simple graph with the vertex set $V$ and the edge set $E$.
For any vertex $v\in V$, let $d(v)$ denote the degree of $v$, i.e.,
the number of vertices adjacent to $v$.
Then, the most straightforward modification of the Zagreb index is to introduce a variable parameter in its definition:  
\begin{align*}
    Z_{G}^{(k)}=\sum_{v\in V} [d(v)]^{k},
\end{align*}
where $k$ is an arbitrary real number.
This generalization of the Zagreb index was first considered in \cite{li2004tree}, 
and then followed by many other authors 
(see, e.g., \cite{zhang2006unicyclic,li2005unified,gutman2014exceptional,bedratyuk2018,gutman2020beyond}).
We refer to an excellent survey \cite{ali2018sum} 
for relations between the generalized Zagreb indices and other topological indices. 
These generalized indices have also been the subject of considerable attention in graph theory
(see, e.g., \cite{Cheng2024fisrt,Cioaba2006sum}). 
For the sake of simplicity, we only consider the case in which $k$ 
is a positive integer in this work.  
When $k=1$, the index $Z_G^{(1)}$ is the total sum of degrees of vertices in $G$.
Then, from the basic knowledge in graph theory,
it is equal to twice the number of edges in $G$.
The case of ordinary Zagreb index is $k=2$. Another important special case is $k=3$, 
where $Z_{G}^{(3)}$ is referred to as the {\em forgotten index} in the literature for historical reasons
(see, e.g., \cite{furtul2015forgotten,jahanbani2021forgortten}).

In the study of random graphs or network data analysis, 
one of the fundamental concepts is the degree distribution (see, e.g., \cite{hofstad2016random}). 
Consider an observed graph on finite vertices generated from a random graph model, 
where the number of vertices could grow to infinity.
From a statistical point of view, 
then the $k$-th order generalized Zagreb index of this graph 
divided by the total number of vertices, 
is an empirical estimation of the $k$-th moment of the 
asymptotic degree distribution of the underlying graph model.
Consequently, the generalized Zagreb indices of various random graph models 
are of their own interest in probability theory and statistics despite of their chemical background.  

The Erd\H{o}s-R\'{e}nyi (ER) random graphs are possibly 
the most extensively studied random graph model in the literature.
We refer the reader to the monographs \cite{janson2000rg,bollobas2001random,hofstad2016random} 
for the enormous results on this classical model.
In this work, we aim to derive the asymptotic properties of
the generalized Zagreb indices of ER random graphs,
as the number of vertices grows to infinity.
In \cite{feng2013zagreb}, several limit laws have been established for 
the ordinary Zagreb index of the ER random graphs, 
as well as for another type of Zagreb index.
Recently, the expectation of the generalized Zagreb index 
of any positive integer order is obtained in \cite{doslic2020general},
also for the ER random graphs. 
For some recent developments of other topological indices with chemical backgrounds in 
the field of random graphs, 
see, for example, \cite{yuan2023asymp, yuan2023randic}.

Throughout this paper, we shall use the following notation.
For an event ${\cal E}$, let $|{\cal E}|$ be the cardinality, 
and $I({\cal E})$ the indicator of ${\cal E}$. 
Let ${\rm Poi}(\lambda)$ denote by the Poisson distribution 
with parameter $\lambda>0$.
For two sequences of positive numbers $a_n$ and $b_n$,
we write $a_n\sim b_n$ for $a_n/b_n\to 1$, $a_n=o(b_n)$  for $a_n/b_n\to 0$, and $a_n=O(b_n)$ for that
$a_n/b_n$ is bounded, as $n\to\infty$.  
For probabilistic convergence, 
we use $\indist$ and $\inprob$ to denote convergence in distribution
and in probability, respectively.

The rest of this paper is organized as follows. 
The generalized Zagreb indices of ER random graphs are formally defined  
in Section 2, and reformulated in terms of the number of star graphs of 
various sizes. Through this reformulation, 
in Section 3 we obtain the first two moments of these indices,
as well as several results of weak convergence.  
In Section 4, under different conditions we establish the joint asymptotic normality of
the first $k$-th order generalized Zagreb indices for any $k\ge2$,
and then provide a necessary and sufficient condition for the $k$-th order generalized Zagreb index 
for any $k\ge1$.

\section{Generalized Zagreb Index and Star Graphs}

Let us denote by $G(n,p)$ an ER random graph on the 
vertex set $\{1,2,\cdots,n\}$, where each possible edge exists independently
with probability $0<p<1$.  
Typically, we consider $p=p(n)$ as a function of $n$ in this work.
For any pair of vertices $(i,j)$, let $I_{ij}$ denote the indicator of  
the event that there exists an edge between $i$ and $j$ in $G(n,p)$. 
Then $I_{ii}=0$, $I_{ij}=I_{ji}$, and $\{I_{ij},1\le i<j\le n\}$
is a sequence of independent Bernoulli variables with common success rate $p$.
For any integer $k\ge 1$, then
the $k$-th order generalized Zagreb index of $G(n,p)$ is formally defined as
\begin{align}\label{defzagreb}
    Z_{n}^{(k)}=\sum_{i=1}^n\bigg(\sum_{j=1}^n I_{ij}\bigg)^k,
\end{align}
where the term $\sum_{j=1}^n I_{ij}$ stands for the degree of vertex $i$.
In particular, the first three indices
$Z_{n}^{(1)},Z_{n}^{(2)}$ and $Z_{n}^{(3)}$ are total sum of vertex degrees,
the ordinary Zagreb index and the forgotten index of $G(n,p)$, respectively. 

For a positive integer $m$, a {\em star graph} of size $m$,  
sometimes simply known as an {\em $m$-star},
is a tree on $m$ vertices with one vertex having degree $m-1$ and the other $m-1$ having degree 1 (see, e.g., \cite{tutte2005graph}). 
In a star graph, we call the vertex with highest degree the {\em center}.
As a special case, a 2-star may have two different centers, 
since both of its vertices have the common degree 1. 
Here, we may distinguish two 2-stars that are defined on the same pair of vertices but have different centers.
As a result, in any graph, the number of 2-stars is equal to twice the number of edges. 
For another special case,  the number of 3-stars corresponds to the number of paths of length two (or wedges) in a graph.
Let $S_{m,n}^{(i)}$ denote the number of $m$-stars 
with the center at vertex $i$ in $G(n,p)$.
That is, for any $1\le i\le n$,
\begin{align}\label{smni}
S_{m,n}^{(i)}=\sum_{1\le j_1<\cdots<j_{m-1}\le n}I_{ij_1}\cdots I_{ij_{m-1}}, \quad m\ge2.
\end{align}
Then, for any integer $m\ge2$, the sum 
\begin{align}\label{smn}
S_{m,n}=\sum_{i=1}^nS_{m,n}^{(i)}
\end{align}
denotes the number of $m$-stars in $G(n,p)$.

The Stirling numbers of the second kind, 
denoted by $\stirling{n}{k}$, are the
numbers of partitions of a set of $n$ elements into $k$ nonempty classes (see, e.g., \cite{wilf2005gf}). 
For any two non-negative integers $n\ge k$, the linear recurrence relation
\begin{equation}\label{stirlingrec}
    \stirling{n}{k}=k\stirling{n-1}{k}+\stirling{n-1}{k-1}
\end{equation}
holds, where, by convention,
\begin{equation*}
    \stirling{0}{0}=1\quad\mbox{and}\quad \stirling{n}{0}=0, \quad n\ge1.
\end{equation*}

The following proposition plays an important role in deriving the asymptotic 
behavior of the generalized Zagreb indices of $G(n,p)$, 
as well as their first two moments. 

\begin{proposition}\label{propstar}
For any $k\ge 1$, the $k$-th order generalized Zagreb index of $G(n,p)$ can be expressed as 
\begin{align}\label{znkstar}
Z_{n}^{(k)}=\sum_{m=1}^k m!\stirling{k}{m}S_{m+1,n}.
\end{align}
\end{proposition}

\begin{proof}
To prove Proposition \ref{propstar}, 
we first show that in a simple graph,
the $k$-th power of the degree of any given vertex $i$
can be expressed in terms of the number of $m$-stars 
with center $i$ for all $2\le m\le k+1$. 
More precisely,  
\begin{equation}\label{kthpowdegree}
    \bigg(\sum_{j=1}^n I_{ij}\bigg)^{k}=\sum_{m=1}^k m!\stirling{k}{m}
    S_{m+1,n}^{(i)},\quad k\ge1.
\end{equation} 
Indeed, we can prove \eqref{kthpowdegree} by induction on $k\ge1$. 
For $k=1$, the both sides of \eqref{kthpowdegree} are the degree of vertex $i$, 
which initializes the induction hypothesis.
To advance the induction hypothesis, 
we suppose that \eqref{kthpowdegree} holds for $k$. 
Recall that $\stirling{k}{0}=0$ and $\stirling{k}{k}=1$ for any $k\ge1$.
Taking into account the number of repetitions, for any given vertex $i$ in $G(n,p)$, we have
\begin{align*}
\sum_{1\le j_1<\dots<j_m\le n}I_{ij_1}\cdots I_{ij_m}\sum_{j\in\{j_1,\cdots,j_m\}} I_{ij}=mS_{m+1,n}^{(i)},
\end{align*}
and 
\begin{align*}
\sum_{1\le j_1<\dots<j_m\le n}I_{ij_1}\cdots I_{ij_m}\sum_{j\notin\{j_1,\cdots,j_m\}} I_{ij}
=\binom{m+1}{m}S_{m+2,n}^{(i)}=(m+1)S_{m+2,n}^{(i)}.
\end{align*}
Then,  by \eqref{smni}, \eqref{stirlingrec} and the induction hypothesis, we obtain that 
\begin{align*}
    \bigg(\sum_{j=1}^n I_{ij}\bigg)^{k+1}
    &=\sum_{m=1}^k m!\stirling{k}{m}S_{m+1,n}^{(i)}\cdot\bigg(\sum_{j=1}^n I_{ij}\bigg)\\
    &=\sum_{m=1}^k m!\stirling{k}{m}\sum_{1\le j_1<\dots<j_m\le n}I_{ij_1}\cdots I_{ij_m}\bigg(\sum_{j\in\{j_1,\cdots,j_m\}} I_{ij}
    +\sum_{j\notin\{j_1,\cdots,j_m\}} I_{ij}\bigg)\\
    &=\sum_{m=1}^km!\stirling{k}{m}\Big(mS_{m+1,n}^{(i)}
     +(m+1)S_{m+2,n}^{(i)}\Big)\\
    &=\sum_{m=1}^km!\bigg(\stirling{k+1}{m}-\stirling{k}{m-1}\bigg)S_{m+1,n}^{(i)}
    +\sum_{m=1}^k(m+1)!\stirling{k}{m}S_{m+2,n}^{(i)}\\
    &=\sum_{m=1}^{k+1} m!\stirling{k+1}{m}S_{m+1,n}^{(i)},
\end{align*}
which advances the induction, 
and thus \eqref{kthpowdegree} holds for any $k\ge1$. 

Using the relation \eqref{kthpowdegree}, it follows by \eqref{defzagreb} that
\begin{align*}
Z_{n}^{(k)}=\sum_{i=1}^n\sum_{m=1}^k m!\stirling{k}{m}S_{m+1,n}^{(i)}
=\sum_{m=1}^k m!\stirling{k}{m}\sum_{i=1}^n S_{m+1,n}^{(i)},
\end{align*}
which, by \eqref{smn}, 
is equal to the right-hand side of \eqref{znkstar}.
This completes the proof of Proposition \ref{propstar}. 
\end{proof}

The above proposition states that in any simple graph, 
the generalized Zagreb index is a linear combination of the numbers of 
star graphs of various sizes,
in which the coefficients involve the Stirling numbers of the second kind.
Specially, it immediately follows from Proposition \ref{propstar} that 
in a simple graph, the ordinary Zagreb index
\begin{equation*}
    Z_n^{(2)}=S_{2,n}+2S_{3,n},
\end{equation*}
which, in fact, has been shown in \cite{feng2013zagreb}, 
and the forgotten index can be expressed as
\begin{equation*}
    Z_n^{(3)}=S_{2,n}+6S_{3,n}+6S_{4,n}.
\end{equation*}

\section{Mean and Variance}

As previously stated in Section 1, the explicit expression of the mean $\E[Z_n^{(k)}]$ is obtained in \cite{doslic2020general} for any integer $k\ge1$.
This is achieved through an approach using the generating polynomial of the degree sequence \cite{sedghi2008poly}. Nevertheless,
in view of Proposition \ref{propstar}, we can directly derive the first two moments of the generalized Zagreb index of $G(n,p)$,  
by considering the numbers of star graphs of sizes not less than 2.
    
For the expectation of $S_{m+1,n}$, it follows by \eqref{smni} and \eqref{smn}  that
\begin{equation}\label{esmn}
    \E [S_{m+1,n}]=n\binom{n-1}{m}p^{m}, \quad 1\le m\le n-1. 
\end{equation}
Let $m$ and $l$ be two fixed integers satisfying $1\le m, l\le n-1$.
Recalling \eqref{smn}, by symmetry it follows that 
the covariance of $S_{m+1,n}$ and $S_{l+1,n}$ is given by
\begin{align}\label{covsmnsln}
\cov(S_{m+1,n},S_{l+1,n})
&=\cov\bigg(\sum_{i=1}^nS_{m+1,n}^{(i)},\sum_{j=1}^nS_{l+1,n}^{(j)}\bigg)\notag\\
&=n\cov\Big(S_{m+1,n}^{(1)},S_{l+1,n}^{(1)}\Big)
  +n(n-1)\cov\Big(S_{m+1,n}^{(1)},S_{l+1,n}^{(2)}\Big).
\end{align}
We next compute the covariances on the right-hand side of 
\eqref{covsmnsln} separately. 
Consider any two possible star graphs with the common center at vertex 1,
but of sizes $m+1$ and $l+1$, respectively. 
Then the covariance of the two products of indicators $I_{1i_1}\cdots I_{1i_{m}}$
and $I_{1j_1}\cdots I_{1j_{l}}$ depends on the number of elements in 
the intersection of sets $\{i_1,\cdots,i_{m}\}$ and $\{j_1,\cdots,j_{l}\}$.
Specially, if the intersection is empty, these two products are independent. 
Furthermore, if there are exactly $s\ge1$ elements in this intersection, 
\begin{align*}
\cov(I_{1i_1}\cdots I_{1i_{m}},I_{1j_1}\cdots I_{1j_{l}})
=p^{m+l-s}-p^{m+l}=p^{m+l-s}(1-p^s),
\end{align*}
where the integer $s$ ranges from $\max\{1,m+l-n+1\}$ to $\min\{m,l\}$.
This implies that
\begin{align}\label{covs1mns1ln}
\cov(S^{(1)}_{m+1,n},S^{(1)}_{l+1,n})
&=\cov\bigg(\sum_{1\le i_1<\cdots<i_{m}\le n}I_{1i_1}\cdots I_{1i_{m}},\sum_{1\le j_1<\cdots<j_{l}\le n}I_{1j_1}\cdots I_{1j_{l}}\bigg)\notag\\
&=\sum_{s=\max\{1,m+l-n+1\}}^{\min\{m,l\}}\binom{n-1}{s,m-s,l-s,n-m-l+s-1}
p^{m+l-s}(1-p^s),
\end{align}
where the multinomial coefficient
\begin{align*}
  \binom{n-1}{s,m-s,l-s,n-m-l+s-1}=
  \frac{(n-1)!}{s!(m-s)!(l-s)!(n-m-l+s-1)!}.
\end{align*}
Note that two star graphs with different centers have at most a common edge,
which connects their centers. This leads to that
\begin{align}\label{covs1mns2ln}
\cov(S^{(1)}_{m+1,n},S^{(2)}_{l+1,n})
&=\cov\bigg(\sum_{3\le i_2<\cdots<i_{m}\le n}I_{12}I_{1i_2}\cdots I_{1i_{m}},\sum_{3\le j_2<\cdots<j_{l}\le n}I_{21}I_{2j_2}\cdots I_{2j_{l}}\bigg)\notag\\
&=\binom{n-2}{m-1}\binom{n-2}{l-1}\cov\bigg(I_{12}I_{13}\cdots I_{1(m+1)},
I_{21}I_{23}\cdots I_{2(l+1)}\bigg)\notag\\
&=\binom{n-2}{m-1}\binom{n-2}{l-1}p^{m+l-1}(1-p).
\end{align}
Plugging \eqref{covs1mns1ln} and \eqref{covs1mns2ln} into \eqref{covsmnsln},
we thus have
\begin{align}\label{covsmnslnex}
\cov(S_{m+1,n},S_{l+1,n})
&=n\sum_{s=\max\{1,m+l-n+1\}}^{\min\{m,l\}} \binom{n-1}{s,m-s,l-s,n-m-l+s-1}\notag\\
&\quad \cdot p^{m+l-s}(1-p^s) +n(n-1)\binom{n-2}{m-1}\binom{n-2}{l-1}p^{m+l-1}(1-p).
\end{align}

For our purposes in this work, we are more interested in the case where $n$ could grow to infinity.
By \eqref{esmn} and \eqref{covsmnslnex} we have that, as $n\to\infty$, for any $m\ge1$,
\begin{align} \label{esm1nasym}
   \E[S_{m+1,n}]&=\frac{n^{m+1}p^m}{m!}\Big(1+O\Big(\frac1n\Big)\Big), 
\end{align}
and for any $m,l\ge1$,
\begin{align}\label{covsm1nsl1nasym}
\cov\big(S_{m+1,n},S_{l+1,n}\big)
=n\bigg(\frac{(np)^{m+l-1}(1-p)}{(m-1)!(l-1)!}+\sum_{s=1}^{\min\{m,l\}}\frac{(np)^{m+l-s}(1-p^s)}{s!(m-s)!(l-s)!}\bigg)\Big(1+O\Big(\frac1n\Big)\Big).
\end{align}
In particular, letting $m=l$ in \eqref{covsm1nsl1nasym} yields that
\begin{align}\label{varsmn}
\V[S_{m+1,n}]=n\bigg(\frac{(np)^{2m-1}(1-p)}{[(m-1)!]^2}+\sum_{s=1}^m\frac{(np)^{2m-s}(1-p^s)}{s![(m-s)!]^2}\bigg)\Big(1+O\Big(\frac1n\Big)\Big), 
\quad m\ge1.
\end{align}

Applying Proposition \ref{propstar}, together with \eqref{esmn} 
and \eqref{covsmnslnex},
it is straightforward to obtain the explicit expressions 
of the mean and variance of $Z_n^{(k)}$, for any $k\ge1$. 
We summarize these results into the following proposition,
and omit the details.

\begin{proposition}\label{propmeanvarznk}
For any $k\ge1$, let $Z^{(k)}_n$ be the $k$-th order generalized Zagreb index of $G(n, p)$.
 Then we have 
\begin{equation}\label{ezk}
    \E \big[Z^{(k)}_n\big]=\sum_{m=1}^k \stirling{k}{m}\frac{n!}{(n-m-1)!}p^m,
\end{equation}
and
\begin{align}\label{vzk}
\V\big[Z^{(k)}_n\big]&=n\sum_{m=1}^k\sum_{l=1}^km!\,l!\stirling{k}{m} \stirling{k}{l}
\bigg[(n-1)\binom{n-2}{m-1}\binom{n-2}{l-1}p^{m+l-1}(1-p)\notag\\ 
&\quad +\sum_{s=\max\{1,m+l-n+1\}}^{\min\{m,l\}} \binom{n-1}{s,m-s,l-s,n-m-l+s-1} p^{m+l-s}(1-p^s) 
\bigg].
\end{align}
\end{proposition}

For any given $k\ge 1$, it is also useful to obtain the asymptotics of the expectation and variance of $Z^{(k)}_n$. 
As $n\to\infty$, it follows that, by \eqref{ezk},
\begin{equation}\label{ezkasymp}
\E\big[Z^{(k)}_n\big]=\sum_{m=1}^k \stirling{k}{m}n^{m+1}p^m\Big(1+O\Big(\frac1n\Big)\Big),
\end{equation}
and that, by \eqref{vzk},
\begin{align}\label{vzkasymp}
\V\big[Z^{(k)}_n\big]&=n\sum_{m=1}^k\sum_{l=1}^km!\,l!\stirling{k}{m}\stirling{k}{l}
\bigg(\frac{(np)^{m+l-1}(1-p)}{(m-1)!(l-1)!}\notag\\ 
&\quad +\sum_{s=1}^{\min\{m,l\}}\frac{(np)^{m+l-s}(1-p^s)}{s!(m-s)!(l-s)!}\bigg)\Big(1+O\Big(\frac1n\Big)\Big).
\end{align}
In particular, for $k=1,2,3$, then \eqref{ezkasymp} and \eqref{vzkasymp} reduce to 
\begin{align*}
\E [Z^{(1)}_n]&=n^2p\Big(1+O\Big(\frac1n\Big)\Big),\\
\V [Z^{(1)}_n]&=2n^2p(1-p)\Big(1+O\Big(\frac1n\Big)\Big);\\
\E [Z^{(2)}_n]&=n^2p(np+1)\Big(1+O\Big(\frac1n\Big)\Big),\\
\V [Z^{(2)}_n]&=2n^2p(1-p)\big(4n^2p^2+5np+1\big)\Big(1+O\Big(\frac1n\Big)\Big);\\
\E [Z^{(3)}_n]&=n^2p\big(n^2p^2+3np+1\big)\Big(1+O\Big(\frac1n\Big)\Big),\\
\V [Z^{(3)}_n]&=2n^2p(1-p)\big(9n^4p^4+45n^3p^3+63n^2p^2+21np+1\big)\Big(1+O\Big(\frac1n\Big)\Big).
\end{align*}

The first two moments of $Z^{(k)}_n$ given in \eqref{ezkasymp} and \eqref{vzkasymp}  lead to the following results on weak convergence,
which generalize Propositions 1 and 2 in \cite{feng2013zagreb} from $k=2$ to general cases.

\begin{proposition}\label{propweakcon}
Let $Z^{(k)}_n$ be the $k$-th order generalized Zagreb index of $G(n, p)$. 
As $n\rightarrow\infty$, the following assertions hold for any $k\ge1$.
\begin{enumerate}[(i)]
\item If $n^2p\rightarrow 0$, then $ Z^{(k)}_n\inprob 0$.
\item If there exists a constant $\lambda>0$ such that $n^2 p\rightarrow\lambda$, 
then $Z^{(k)}_n/2 \indist {\rm Poi}(\lambda/2)$.
\item If $n^2 p\rightarrow \infty$, then $Z^{(k)}_n/ \E [Z^{(k)}_n]\inprob 1$.
\item If $n^2(1-p)\rightarrow 0$, then $\Prob(Z^{(k)}_n=n(n-1)^k)\rightarrow 1$.
\item If there exists a constant $\lambda>0$ such that $n^2 (1-p)\rightarrow \lambda$, then
$$\frac{n(n-1)^k-Z^{(k)}_n}{2kn^{k-1}} \indist {\rm Poi}(\lambda/2).$$
\item  If $n^2 (1-p)\rightarrow \infty$, then
$$\frac{n(n-1)^k-Z^{(k)}_n}{n^{k+1}(1-p^k)}\inprob 1.$$
\end{enumerate}
\end{proposition}

\begin{proof} We prove the assertions in Proposition \ref{propweakcon} in turn.
If $n^2 p\rightarrow 0$, by \eqref{ezkasymp} and \eqref{vzkasymp},
both $\E [Z^{(k)}_n]$ and $\V [Z^{(k)}_n]$ tend to 0 for any $k\ge 1$,
which implies (i).
In fact, under the condition $n^2 p\rightarrow 0$, the probability that
$G(n,p)$ is an empty graph tends to 1.

Assume that $n^2 p\rightarrow \lambda>0$. 
Then, for any $m\ge2$, it follows that $n^{m+1}p^m\to 0$.
This implies that for any $m\ge2$, 
both $\E [S_{m+1,n}]$ 
and $\V [S_{m+1,n}]$ tend to 0, by \eqref{esm1nasym} and \eqref{varsmn}.
Therefore, similarly to (i), we can obtain that $S_{m+1,n}$ converges 
in probability to 0 for any $m\ge2$. This implies that,
by Proposition \ref{propstar}, for any $k\ge2$, the index $Z^{(k)}_n$ has
the same asymptotic behavior as $S_{2,n}$, 
which is equal to twice the number of edges in $G(n,p)$.
To prove (ii), it is sufficient to show that $S_{2,n}/2$ converges 
in distribution to ${\rm Poi}(\lambda/2)$, which, in fact, is already known 
(see, e.g., Theorem 3.19 in \cite{janson2000rg}).

If $n^2 p\rightarrow \infty$, it follows by \eqref{ezkasymp} and \eqref{vzkasymp}
that $\V [Z^{(k)}_n]=o((\E [Z^{(k)}_n])^2)$ whether $np$ converges or not.
Then (iii) holds by Chebyshev’s inequality. 

To prove (iv) and (v), we shall consider $\overline{G}(n, p)$, 
the complement graph of $G(n,p)$.
Recall that the graphs $\overline{G}(n, p)$ and $G(n,p)$ 
have the same vertex set $\{1,2,\cdots,n\}$,
but any two distinct vertices are adjacent in $\overline{G}(n, p)$ 
if and only if they are not adjacent in $G(n,p)$.
For any two vertices $1\le i\neq j\le n$,
denote by $\overline{I}_{ij}$ the indicator of that there exists between
an edge $i$ and $j$ in $\overline{G}(n, p)$. 
Then, we have that $\overline{I}_{ii}=0$, and
\begin{align*}
\overline{I}_{ij}+I_{ij}=1, \quad  1\le i\neq j\le n. 
\end{align*}

If $n^2(1-p) \rightarrow 0$, by symmetry, 
(i) implies that the probability that $\overline{G}(n, p)$ is empty tends to 1,
and thus $G(n,p)$ is a complete graph with probability tending to 1.
Then (iv) holds by the fact that the $k$-th order generalized Zagreb index of  
a complete graph on $n$ vertices is $n(n-1)^k$ for any $k\ge1$.

In the complement graph $\overline G(n, p)$,
we also define $\overline{S}_{m,n}$ to be the number of $m$-stars for any $m\ge1$,
and $\overline{Z}_{n}^{(k)}=\sum_{i=1}^n(\sum_{j=1}^n \overline I_{ij})^k$ to be the $k$-th order generalized Zagreb index.
Note that
\begin{align}\label{znkcomp}
Z_{n}^{(k)} &= \sum_{i=1}^n\bigg(n-1-\sum_{j=1}^n \overline{I}_{ij}\bigg)^k \notag\\
 &= n(n-1)^k+\sum_{i=1}^n\sum_{m=1}^k (-1)^m \binom{k}{m}(n-1)^{k-m}\bigg(\sum_{j=1}^n\overline{I}_{ij} \bigg)^m\notag\\
 &= n(n-1)^k+\sum_{m=1}^k (-1)^m \binom{k}{m}(n-1)^{k-m}\overline{Z}_{n}^{(m)}
\end{align}
and 
\begin{align}\label{idencomb}
-\sum_{m=1}^k (-1)^m \binom{k}{m}(n-1)^{k-m}=(n-1)^k-(n-2)^k=kn^{k-1}\Big(1+O\Big(\frac1n\Big)\Big).
\end{align}
By Proposition \ref{propstar}, we have
\begin{align*}
\overline{Z}_{n}^{(k)}=\sum_{m=1}^k m!\stirling{k}{m}\overline{S}_{m+1,n}.
\end{align*}
Therefore, if $n^2(1-p)\to\lambda$, from the proof of (ii) it follows that 
$\Prob(\overline{Z}_{n}^{(k)}=\overline{S}_{2,n})\to1$  for any given $k\ge1$,
and $\overline{S}_{2,n}/2\inprob {\rm Poi}(\lambda/2)$.
Hence, the assertion (v) holds by \eqref{znkcomp} and \eqref{idencomb}.

If $n^2 (1-p)\rightarrow \infty$, it follows that, by \eqref{ezkasymp}, 
\begin{align*}
    n(n-1)^k-\E[Z^{(k)}_n]&=n(n-1)^k-\sum_{m=1}^k \stirling{k}{m}n^{m+1}p^m\\
    &=n^{k+1}(1-p^k)\Big(1+O\Big(\frac1n\Big)\Big).
\end{align*}
and that, by \eqref{vzkasymp},
\begin{align*}
\V [n(n-1)^k-Z^{(k)}_n]=\V[Z^{(k)}_n]
    =O(n^{2k}(1-p))
    =o\big(\E\big[n(n-1)^k-Z^{(k)}_n\big]^2\big).
\end{align*}
This implies that (vi) holds by Chebyshev’s inequality, 
and thus completes the proof of Proposition \ref {propweakcon}. 
\end{proof}

\section{Asymptotic Normality of Generalized Zagreb Indices}
In this section, we shall establish the asymptotic normality of 
the generalized Zagreb indices of $G(n,p)$. 
As in the previous section, our methodology depends heavily on the limiting behaviors 
of the numbers of star graphs of various sizes. Indeed,  
there is a substantial body of research on the asymptotic properties of subgraph counts in $G(n,p)$
that has been conducted over several decades (see \cite{janson2000rg} and references therein).
A necessary and sufficient condition for asymptotic normality of the number of a given subgraph was first given by Ruci\'nski
\cite{Rucinski1998when}. More precisely, for any fixed simple graph $\mathbb{H}$, 
let $H_n=H_n(\mathbb{H})$ denote the number of subgraphs of $G(n,p)$  that are isomorphic to $\mathbb{H}$.
As $n\to\infty$, the normalized random variable $(H_n-\E[H_n])/\sqrt{{\rm Var}[H_n]}$
 converges in distribution to the standard normal distribution $N(0,1)$ if and only if
\[
\min_{\mathbb{G}\subset \mathbb{H}:\, e_{\mathbb{G}}\ge1}\big\{n^{v_\mathbb{G}}p^{e_\mathbb{G}}\big\}\to\infty \quad \mbox{and} \quad
n^2(1-p)\to\infty,
\]
 where $v_\mathbb{G}$ and $e_\mathbb{G}$ stand for the numbers of vertices and edges of $\mathbb{G}$, respectively. 
Especially, for the number of $m$-stars $S_{m,n}$ in $G(n,p)$ with any fixed $m\ge2$, we have that 
 $(S_{m,n}-\E[S_{m,n}])/\sqrt{{\rm Var}[S_{m,n}]}$  converges in distribution to  $N(0,1)$ if and only if one of the following three conditions holds:
 (1) $np\to 0$ and $n^{m+1}p^m\to\infty$, (2) $np\to c$ for some constant $c > 0$, (3) $np\to\infty$ and $n^2(1-p)\to\infty$. 
Until recently,  Ruci\'nski's result was complemented with explicit bounds 
on the Kolmogorov distance in \cite{privault2020normal}, see also \cite{eichel2023Kol,rollon2022Kol}.
When $0<p<1$ is fixed, the joint normality of the numbers of several distinct subgraphs is shown in 
\cite{janson1991asymp,reinert2010random}, with a corresponding result on the convergence rate in 
\cite{krokowski2017multi}.  
In addition to ER random graphs, the problems of subgraph counting are also considered in other random graph models, 
such as the Poisson random connection model \cite{liu2024+normal,penrose2018inhomo}, the graphon-based random graphs 
\cite{kaur2021higher, bhatt2023fluc}, the uniform attachment model \cite{bjo2024appr},
and the random hypergraphs \cite{dejong1996central,michalczuk2024+normal}.

The following theorem establishes the joint asymptotic normality of the numbers of $k$ different stars
under suitable conditions where $p$ is allowed to be vary with $n$.

\begin{theorem}\label{thmstar}
For any $m\ge 2$, let $S_{m,n}$ be the number of $m$-stars in $G(n,p)$.
For any given integer $k\ge 1$, as $n\rightarrow \infty$, the following assertions hold.
\begin{enumerate}[(i)]
\item If $np\rightarrow c$ for some constant $c > 0$, then
\begin{equation*}
    \frac{1}{\sqrt{n}}\bigg(
        S_{2,n}-\E[S_{2,n}],S_{3,n}-\E[S_{3,n}], \cdots, S_{k+1,n}-\E[S_{k+1,n}]
    \bigg)^{\top}\indist N\big(\bm{0},\bm{\Sigma}_{k}^{\star}\big),
\end{equation*}
where $\bm{0}$ is a $k$-dimensional vector of zeros, 
and the $(m,l)$-entry of the covariance matrix 
$\bm{\Sigma}_{k}^{\star}$ is given by
\begin{equation*}
    \sigma_{ml}^{\star}=\frac{c^{m+l-1}}{(m-1)!(l-1)!} +\sum_{s=1}^{\min(m,l)}\frac{c^{m+l-s}}{s!(m-s)!(l-s)!}, \quad 1\le m,l\le k.
\end{equation*}

\item If $np\rightarrow \infty$ and $n^2 (1-p)^3\rightarrow \infty$, then,
\begin{equation*}
    \sqrt{\frac{p}{2(1-p)}}
    \Bigg(
        \frac{S_{2,n}-\E[S_{2,n}]}{np},\frac{S_{3,n}-\E[S_{3,n}]}{(np)^2},\cdots,\frac{S_{k+1,n}-\E[S_{k+1,n}]}{\frac{1}{(k-1)!}(np)^k}
    \Bigg)^{\top}\indist N\big(\bm{0},\bm{U}_{k}\big),
\end{equation*}
where $\bm{U}_{k}$ is a matrix of order $k$ with all entries equal to 1.
\end{enumerate}
\end{theorem}

It is noteworthy that in both (i) and (ii) of Theorem \ref{thmstar}, 
the expectation $\E[S_{m+1,n}]$ cannot be replaced by its leading term for any $1\le m\le k$,
since the Slutsky's Theorem cannot be directly applied here.
The reason is as follows. 
We first consider (ii) in the special case where $p\in(0,1)$ is a constant. 
By \eqref{esm1nasym} and \eqref{varsmn},  we now have that the difference 
$\E[S_{m+1,n}]-n^{m+1}p^m/m!=O(n^m)$
is of the same order as the standard deviation of $S_{m+1,n}$. 
Then, in Case (i),  also by \eqref{esm1nasym} and \eqref{varsmn} we have that for any $m\ge1$,
\begin{align}
	\E[S_{m+1,n}]&\sim \frac{c^{m}}{m!}n, \label{esmnasymi}\\
	\V[S_{m+1,n}]&\sim \bigg(\frac{c^{2m-1}}{[(m-1)!]^2}+\sum_{s=1}^{m}\frac{c^{2m-s}}{s![(m-s)!]^2}\bigg)n, 
	\label{varsmnasymi}
\end{align}
both of which are of the same order $n$. 
At this time, we choose the probabilities $p=c(1+a_n)^{1/m}/n$, 
such that the sequence $\{a_n, n\ge1\}$ satisfies $a_n\to 0$ and $\lim\sqrt{n}\,a_n>0$. 
After straightforward calculations, by \eqref{esmn} one can obtain
\[
\E[S_{m+1,n}]-n{c^{m}}/{m!}=n\binom{n-1}{m}p^m-n{c^{m}}/{m!}=
\frac{c^{m}n}{m!}\Big(a_n+O\Big(\frac1n\Big)\Big).
\]
Therefore, we cannot assert that the difference $\E[S_{m+1,n}]-nc^m/m!$ is of order $o(\sqrt{n})$, 
due to the assumption that $\sqrt{n}\,a_n$ converges to a positive limit.

Before proving Theorem \ref{thmstar}, we shall briefly introduce 
the concept of dependency graph for random variables \cite{janson2000rg}.
Let $\{X_i\}_{i\in {\cal I}}$ be a family of random variables 
defined on a common probability space, where ${\cal I}$ is an index set.
A {\em dependency graph} for $\{X_i\}_{i\in {\cal I}}$ is any graph $L$ with the vertex set ${\cal I}$,
such that for any two disjoint subsets ${\cal A}, {\cal B}\subset {\cal I}$,
if there exist no edges between them, 
the families $\{X_i\}_{i\in {\cal A}}$  and $\{X_j\}_{j\in {\cal B}}$
are mutually independent.
For any integer $m\ge1$ and $i_1,i_2,\cdots,i_m\in {\cal I}$, let
\begin{align}\label{closedn}
\overline{N}_L(i_1,i_2,\cdots,i_m)=\bigcup_{s=1}^m\{j\in{\cal I}: j=i_s~\mbox{or there exists an edge between}~j~\mbox{and}~i_s~\mbox{in}~L\} 
\end{align}
denote the closed neighborhood of vertices $i_1,i_2,\cdots,i_m$ in $L$.

The following auxiliary lemma, stated as Theorem 6.33 in \cite{janson2000rg}, 
is a particularly useful tool to determine the asymptotic normality of the sum of 
a family of dependent random variables.

\begin{lemma}\label{lemjanson}
 Let $\{Y_n\}_{n=1}^{\infty}$ be a sequence of random variables such that 
 $Y_n=\sum_{\alpha\in {\cal A}_n}X_{n\alpha}$, where for each $n$, the random variables 
 $\{X_{n\alpha}: \alpha\in  {\cal A}_n\}$ has a dependency graph $L_n$.
 Suppose that there exist numbers $M_n$ and $Q_n$ such that 
 $\sum_{\alpha\in {\cal A}_n}\E[|X_{n\alpha}|]\le M_n$,
 and for every $\alpha_1,\alpha_2$,
 \begin{align*}
     \sum_{\alpha\in \overline{N}_{L_n}(\alpha_1,\alpha_2)}\E\big[|X_{n\alpha}|\big|X_{n\alpha_1},X_{n\alpha_2}\big]\le Q_n.
 \end{align*}
 As $n\to \infty$, if $M_nQ_n^2/(\V[Y_n])^{3/2} \to 0$, then
 \begin{align*}
     \frac{Y_n-\E[Y_n]}{\sqrt{\V[Y_n]}}\indist N(0,1).
 \end{align*}
\end{lemma}

We now give a formal proof of Theorem \ref{thmstar} in the following. 

\begin{proof}[Proof of Theorem \ref{thmstar}]
If $np\to c$ for some constant $c>0$, it follows by \eqref{covsm1nsl1nasym} that 
%by \eqref{esm1nasym} and \eqref{varsmn}, 
%\begin{align}
%\E[S_{m+1,n}]&=\frac{c^{m}n}{m!}\big(1+O\big(n^{-1}\big)\big), \label{esmnasymi}\\
%\V[S_{m+1,n}]&=\bigg(\frac{c^{2m-1}}{[(m-1)!]^2}+\sum_{s=1}^{m}\frac{c^{2m-s}}{s![(m-s)!]^2}\bigg)n\Big(1+O\Big(\frac1n\Big)\Big), 
%\label{varsmnasymi}
%\end{align}
%for $m\ge1$, and by \eqref{covsm1nsl1nasym}, 
\begin{equation}\label{covsmnasym}
\lim_{n\to\infty}\cov\bigg(\frac{S_{m+1,n}}{\sqrt{n}},\frac{S_{l+1,n}}{\sqrt{n}}\bigg)=\frac{c^{m+l-1}}{(m-1)!(l-1)!}+
\sum_{s=1}^{\min\{m,l\}}\frac{c^{m+l-s}}{s!(m-s)!(l-s)!}, 
\end{equation}
for $ m,l\ge1$. Then, by \eqref{covsmnasym}, the limit of the covariance matrix $\cov(\bm{S}_{k,n}/\sqrt{n})$, 
where the random vector 
\begin{align*}
 \bm{S}_{k,n}:=\big(S_{2,n}, S_{3,n},\cdots,S_{k+1,n}\big)^{\top},
\end{align*}
is given by
\begin{align}\label{limcovmat}
\bm{\Sigma}_{k}^{\star}=\left(
    \begin{array}{cccc}
       2c  & \frac{2c^2}{1!}  & \cdots & \frac{2c^k}{(k-1)!}\\
       \frac{2c^2}{1!}  & \frac{2c^3}{1!}+\frac{c^2}{2!} & \cdots & \frac{2c^{k+1}}{(k-1)!}+\frac{c^k}{2!(k-2)!}\\
       \frac{2c^3}{2!}  & \frac{2c^4}{2!}+\frac{c^3}{2!1!} & \cdots & \frac{2c^{k+2}}{2!(k-1)!}+\frac{c^{k+1}}{2!(k-2)!}+\frac{c^k}{3!(k-3)!}\\
       \vdots & \vdots & \ddots & \vdots\\
       \frac{2c^k}{(k-1)!} & \frac{2c^{k+1}}{(k-1)!}+\frac{c^k}{2!(k-2)!} & \cdots &  \frac{c^{2k-1}}{[(k-1)!]^2}
       +\sum_{s=1}^k\frac{c^{2k-s}}{s![(k-s)!]^2}
       \\
    \end{array}
    \right).
\end{align}
For all $1\le m<k$ and for all $m<l\le k$, by adding $-c^{l-m}/(l-m)!$ times the $m$-th column to the $l$-th column, 
we can obtain a triangular matrix with diagonal elements $2c, c^2/2!, c^3/3!,\cdots, c^k/k!$. Hence,
the determinant of $\bm{\Sigma}_{k}^{\star}$ is 
\begin{equation*}%\label{detsigmak}
    \det \big(\bm{\Sigma}_{k}^{\star}\big)=2c^{\frac{k(k+1)}{2}}\prod_{m=1}^{k}\frac{1}{m!},
\end{equation*}
which implies that $\bm{\Sigma}_{k}^{\star}$ is a positive definite matrix for any $k\ge1$.
Consequently, there exists a real number $\lambda_{\min}>0$, which only depends on $c$ and $k$, such that
$\lambda_{\min}$ is the smallest eigenvalue of $\bm{\Sigma}_k^{\star}$. 

To prove (i), we use the well-known Cram\'er-Wold Theorem.
 For any given integer $k\ge1$, consider an arbitrary linear combination in the form
\begin{equation}\label{tknlinear}
    T_{k,n}=\sum_{m=1}^{k} a_{m} S_{m+1,n},
\end{equation}
where $a_1, a_2,\cdots,a_k$ are real numbers such that $\sum_{m=1}^k a_m^2=1$.
Then it holds that $|a_m|\le 1$ for all $1\le m\le k$. 
%Without loss of generality, we further assume that $a_k\neq 0$. 
Then, by Slutsky's theorem together with \eqref{varsmnasymi}, it is sufficient to show that
\begin{equation}\label{tmnindist}
    \frac{T_{k,n}-\E [T_{k,n}]}{\sqrt{\V [T_{k,n}]}}\indist N(0,1).
\end{equation}

To apply Lemma \ref{lemjanson}, 
let us denote by $\{\mathbb{S}_{n\alpha}\}_{\alpha \in {\cal B}_n}$ the set of 
all possible $(m+1)$-stars in $G(n,p)$ for $1\le m\le k$,  
where ${\cal B}_n$ is an index set with the cardinality
\begin{equation}\label{calBn}
    |{\cal B}_n|=n\sum_{m=1}^{k}\binom{n-1}{m}.
\end{equation}
For any $\alpha \in {\cal B}_n$, let $s_{n\alpha}$ be the size of $\mathbb{S}_{n\alpha}$ (i.e., the number of vertices in $\mathbb{S}_{n\alpha}$), 
\begin{equation}\label{xnadef}
I_{n\alpha}=I(\mathbb{S}_{n\alpha}\mbox{~is contained in~}G(n,p)) \quad \mbox{and} \quad X_{n\alpha}=a_{s_{n\alpha}-1}I_{n\alpha}. 
\end{equation}
Then, the sum $T_{k,n}$ given in \eqref{tknlinear} can be rewritten as
\begin{equation*}
    T_{k,n}=\sum_{\alpha\in {\cal B}_n}X_{n\alpha}.
\end{equation*}
Further, we can construct a dependency graph $L_n$ with the vertex set ${\cal B}_n$ for 
the random variables $\{X_{n\alpha}:\alpha\in {\cal B}_n\}$ as follows.
In this context, all the $(m+1)$-stars with $1\le m\le k$ in $\mathbb{K}_n$, the complete graph with vertex set $\{1,2,\cdots,n\}$, 
are now regarded as the new ``vertices" in $L_n$. Then the number of vertices in $L_n$ is $|{\cal B}_n|$, which is given by \eqref{calBn}.
In order to form the edge set of $L_n$,
for each pair of vertices $\alpha$ and $\beta$ in ${\cal B}_n$ we connect them by an edge, 
if their corresponding star graphs $\mathbb{S}_{n\alpha}$ and $\mathbb{S}_{n\beta}$ have at least one common edge in $\mathbb{K}_n$.

We now verify the condition in Lemma \ref{lemjanson}. 
By definition it is clear to see that
\begin{equation}\label{sumexalpha}
    \sum_{\alpha\in {\cal B}_n}\E [|X_{n\alpha}|]\le \sum_{m=1}^{k} \E[S_{m+1,n}].
\end{equation}
Then, by \eqref{esmnasymi} and \eqref{sumexalpha} we have that when $n$ is sufficiently large, 
\begin{equation}\label{mnorderi}
   \sum_{\alpha\in {\cal B}_n}\E [|X_{n\alpha}|]\le  M_n:=2n\sum_{m=1}^{k}\frac{c^m}{m!},
\end{equation}
which is of order $n$.
For any given $\alpha_1,\alpha_2 \in {\cal B}_n$, 
let $V_n=V_n(\alpha_1,\alpha_2)$ be the vertex set of the subgraph $\mathbb{S}_{n\alpha_1}\cup \mathbb{S}_{n\alpha_2}$. 
Since the size of $\mathbb{S}_{n\alpha}$ is at most $k+1$ for any $\alpha\in {\cal B}_n$,
the number of vertices in $V_n$ does not exceed $2(k+1)$. 
Recall the definition of the closed neighborhood of given vertices in a dependency graph, as given in \eqref{closedn}.
Then $\overline N_{L_n}(\alpha_1,\alpha_2)$, which is a subset of ${\cal B}_n$, represents
all the star graphs of size at most $k+1$ in $\mathbb{K}_n$ that share at least one edge with either $\mathbb{S}_{n\alpha_1}$ or $\mathbb{S}_{n\alpha_2}$. 
For any $\alpha\in \overline N_{L_n}(\alpha_1,\alpha_2)$, 
consider the number of vertices in $\mathbb{S}_{n\alpha}$ that are not in $V_n$. Denote
\begin{equation*}
    \overline N_s=\big\{\alpha\in \overline N_{L_n}(\alpha_1,\alpha_2): \text{ $\mathbb{S}_{n\alpha}$ has exactly $s$ vertices out of $V_n$}\big\},
\end{equation*}
for $s=0,1,\dots,k-1$. Then $\bigcup_{s=0}^{k-1}\overline N_s=\overline N_{L_n}(\alpha_1,\alpha_2)$, and 
for any $0\le s\le k-1$ there exists a sufficiently large number $C_k$,
which does not depend on $n$ but on $k$,
such that the number of star graphs in $\overline N_s$ is not greater than $C_k\binom{n-2}{s}$.
Since for any $\alpha\in\overline N_s$,
\begin{equation*}
    \E \big[|X_{n\alpha}|\big|X_{n\alpha_1},X_{n\alpha_2}\big]\le p^s,
\end{equation*}
we have
\begin{equation}\label{sumexalphacond}
    \sum_{\alpha\in \overline N_s}\E\big[|X_{n\alpha}|\big|X_{n\alpha_1},X_{n\alpha_2}\big]
    \le C_k\binom{n-2}{s}p^s, \quad  0\le s\le k-1,
\end{equation}
for sufficiently large $n$.
Then, by \eqref{sumexalphacond} we can set $Q_n$ to be 
\begin{equation}\label{qnorderi}
Q_n=2C_k\sum_{s=0}^{k-1}\frac{c^s}{s!},
\end{equation}
which is a constant, regardless of $n$. 
Recall $\lambda_{\min}>0$ is the smallest eigenvalue of the limiting covariance matrix
$\bm{\Sigma}_k^{\star}$ given in \eqref{limcovmat}.
 Then it follows that for sufficiently large $n$, 
\begin{equation*}
    \V \bigg(\frac{T_{k,n}}{\sqrt{n}}\bigg)
    =\bm{a}^{\top}\cov\Big(\frac{\bm{S}_{k,n}}{\sqrt{n}}\Big)\bm{a}\ge \lambda_{\min}(1+o(1)),
\end{equation*}
where $\bm{a}$ denotes the vector $(a_1,a_2,\cdots,a_k)^{\top}$.
This implies that the variance of $T_{k,n}$ is of order $n$.
Collecting the orders of magnitude of $M_n$ in \eqref{mnorderi} and $Q_n$ in \eqref{qnorderi}, 
we thus have 
\begin{equation*}
    \frac{M_nQ_n^2}{(\V [T_{k,n}])^{3/2}}=O\Big(\frac1{\sqrt n}\Big)\to 0.
\end{equation*}
This proves \eqref{tmnindist} by Lemma \ref{lemjanson}, and thus (i) holds.

Next we prove (ii). 
If $np\rightarrow \infty$ and $n^{2}(1-p)^3\rightarrow \infty$,
by \eqref{covsm1nsl1nasym} and \eqref{varsmn} we have that
\begin{align}\label{covsasymnpinf}
\cov(S_{m+1,n},S_{l+1,n})=\frac{2}{(m-1)!(l-1)!}n^{m+l}p^{m+l-1}(1-p)\Big(1+O\Big(\frac1{np}\Big)\Big), \quad m,l\ge1,
\end{align}
and 
\begin{align}\label{varsasymnpinf}
\V[S_{m+1,n}]=\frac{2}{[(m-1)!]^2}n^{2m}p^{2m-1}(1-p)\Big(1+O\Big(\frac1{np}\Big)\Big), \quad m\ge1.
\end{align}
Denote
\begin{align*}
    S_{m+1,n}':=\frac{S_{m+1,n}}{\sqrt{n^{2m}p^{2m-1}(1-p)}}, \quad 1\le m\le k.
\end{align*}
Then the linear combination 
\begin{align*}
T_{k,n}':=\sum_{m=1}^ka_mS_{m+1,n}'=\sum_{\alpha\in {\cal B}_n}X_{n\alpha}',
\end{align*}
where $a_1, a_2,\cdots,a_k$ are arbitrary real numbers such that $\sum_{m=1}^k a_m^2=1$, and
\begin{align*}
    X'_{n\alpha}=\frac{X_{n\alpha}}{\sqrt{n^{2s_{n\alpha}-2}p^{2s_{n\alpha}-3}(1-p)}},\quad  \alpha \in {\cal B}_n,
\end{align*}
with random variables $X_{n\alpha}$ given by \eqref{xnadef}.

By \eqref{varsasymnpinf}, we have that $\V [T'_{k,n}]$ is of order 1.
%Analogously to \eqref{mnorderi} and \eqref{qnorderi}, 
It follows by \eqref{esm1nasym} that we can set
\begin{align*}
M_n'&:=\sum_{m=1}^{k} \E[S_{m+1,n}']=O\Big(n\sqrt{\frac{p}{1-p}}\Big).
\end{align*}
For any $1\le s\le k-1$, it follows by \eqref{sumexalphacond} that 
\begin{align*}
\sum_{\alpha\in \overline N_s}\E\big[|X_{n\alpha}'|\big|X_{n\alpha_1}',X_{n\alpha_2}'\big]
& = \sqrt{\frac{p}{1-p}}\sum_{\alpha\in \overline N_s}\frac{\E\big[|X_{n\alpha}|\big|X_{n\alpha_1},X_{n\alpha_2}\big]}{(np)^{s_{n\alpha}-1}}\\
& \le \frac{1}{(np)^{s+1}}\sqrt{\frac{p}{1-p}}\sum_{\alpha\in \overline N_s}\E\big[|X_{n\alpha}|\big|X_{n\alpha_1},X_{n\alpha_2}\big]\\
& =O\Big(\frac{1}{n\sqrt{p(1-p)}}\Big).
\end{align*}
Analogously to \eqref{qnorderi}, for every $\alpha_1,\alpha_2 \in {\cal B}_n$, then there exists a constant $Q_n'$ such that
\begin{align*}
\sum_{\alpha\in \overline N_L(\alpha_1,\alpha_2)}\E\big[|X_{n\alpha}'|\big|X_{n\alpha_1}',X_{n\alpha_2}'\big]=
\sum_{s=0}^{k-1}\sum_{\alpha\in \overline N_s}\E\big[|X_{n\alpha}'|\big|X_{n\alpha_1}',X_{n\alpha_2}'\big]\le Q_n'
=O\Big(\frac{1}{n\sqrt{p(1-p)}}\Big).
\end{align*}
Thus, we have 
\begin{equation*}
    \frac{M_n'(Q_n')^2}{(\V [T_{k,n}'])^{3/2}}=O\bigg(\frac{1}{\sqrt{n^{2}p(1-p)^3}}\bigg).
\end{equation*}
Note that now the correlation coefficient of $S_{m+1,n}$  and $S_{l+1,n}$ tends to 1 whenever $np\to\infty$, 
by \eqref{covsasymnpinf} and \eqref{varsasymnpinf}. 
Since $np\rightarrow \infty$ and $n^{2}(1-p)^3\rightarrow \infty$ imply $n^{2}p(1-p)^3\to \infty$,
it follows by Lemma \ref{lemjanson} that (ii) holds, and completes the proof of Theorem \ref{thmstar}.
\end{proof}

For any given integer $k\ge2$, Theorem \ref{thmstar} gives us a tool to study the joint asymptotic normality 
of the first $k$ generalized Zagreb indices, as $n\to\infty$. 
Let $\bm{Z}_{k,n}$ denote the random vector $(Z_n^{(1)}, Z_n^{(2)},\cdots,Z_n^{(k)})^{\top}$.
It thus follows from Proposition \ref{propstar} that
\begin{align*}
\bm{Z}_{k,n}={\bm A}_k\bm{S}_{k,n},
\end{align*}
where ${\bm A}_k=(a_{ml})_{k\times k}$ is a lower triangular matrix,
which does not depend on $n$, with 
\begin{align*}
a_{ml}=\left\{\begin{array}{cc}
   l!\stirling{m}{l},  &   m\ge l; \\[2pt]
    0,  &  m<l.
\end{array}\right.
\end{align*}
Therefore, under the same conditions in Theorem \ref{thmstar}, 
the random vector $\bm{Z}_{k,n}$ has asymptotic normality
when suitably normalized. 

\begin{theorem}\label{thmzagreb}
For any $m\ge 1$, let $Z^{(m)}_n$ be the $m$-th order generalized Zagreb index of $G(n, p)$.
For any given integer $k\ge 2$, as $n\rightarrow \infty$, the following assertions hold.
\begin{enumerate}[(i)]
\item If $np\rightarrow c$ for some constant $c > 0$, then
\begin{equation*}
    \frac{1}{\sqrt{n}}\big(Z^{(1)}_n-\E[Z^{(1)}_n],Z^{(2)}_n-\E[Z^{(2)}_n], \cdots, Z^{(k)}_n-\E[Z^{(k)}_n]
    \big)^{\top}\indist N(\bm{0},\bm{\Sigma}_k),
\end{equation*}
where the covariance matrix $\bm{\Sigma}_k=(\sigma_{ml})_{k\times k}$ is given by
\begin{equation*}
    \sigma_{ml}=\sum_{i=1}^{m}\sum_{j=1}^{l}\stirling{m}{i}\stirling{l}{j}
    \bigg[ij c^{i+j-1}+\sum_{s=1}^{\min\{i,j\}}\frac{i!\,j!\, c^{i+j-s}}{s!(i-s)!(j-s)!}\bigg], \quad 1\le m,l\le k.
\end{equation*}

\item If $np\rightarrow \infty$ and $n^2 (1-p)^3\rightarrow \infty$, then,
\begin{equation*}
    \sqrt{\frac{p}{2(1-p)}}
    \bigg(
        \frac{Z^{(1)}_n-\E[Z^{(1)}_n]}{np},\frac{Z^{(2)}_n-\E[Z^{(2)}_n]}{2(np)^2},\cdots,\frac{Z^{(k)}_n-\E[Z^{(k)}_n]}{k(np)^k}
    \bigg)^{\top}\indist N\big(\bm{0},\bm{U}_{k}\big).
\end{equation*}

\item If $np\rightarrow 0$ and $n^{2}p\rightarrow \infty$, then
\begin{equation*}
    \frac{1}{\sqrt{2n^2p}}\bigg(Z^{(1)}_n-n^2p,Z^{(2)}_n-n^3p^2-n^2p,\cdots,Z^{(k)}_n-\sum_{m=1}^k 
    \stirling{k}{m}n^{m+1}p^m\bigg)^{\top}\indist N\big(\bm{0},\bm{U}_{k}\big).
\end{equation*}
\end{enumerate}
\end{theorem}

\begin{proof}
Similarly to Proposition \ref{propmeanvarznk}, for any integers $m,l\ge1$,
after straightforward calculations we can first obtain that the asymptotic covariance
of $Z^{(m)}_n$ and $Z^{(l)}_n$ is given by 
\begin{align}\label{covzk1zk2asymp}
\cov \big(Z^{(m)}_n,Z^{(l)}_n\big)&=n\sum_{i=1}^{m}\sum_{j=1}^{l}\stirling{m}{i}\stirling{l}{j}
\bigg[ij(np)^{i+j-1}(1-p) \notag\\
&\quad +\sum_{s=1}^{\min\{i,j\}}
 \frac{i!\,j!}{s!(i-s)!(j-s)!}(np)^{i+j-s}(1-p^s) \bigg]\Big(1+O\Big(\frac1n\Big)\Big). 
\end{align}

We first consider (i). Since Theorem \ref{thmstar}, together with Proposition \ref{propstar}, guarantees the 
asymptotic normality of the $k$-dimensional random vector $(\bm{Z}_{k,n}-\E [\bm{Z}_{k,n}])/\sqrt{n}$,
by Slutsky's theorem it is sufficient to check its asymptotic covariance matrix. 
Thus, (i) follows directly by \eqref{covzk1zk2asymp}.

The proof of (ii) is similar to that of (i). By \eqref{vzkasymp}, we have that $\V[Z^{(k)}_n]\sim 2k^2n^{2k}p^{2k-1}(1-p)$
for any $k\ge1$. It thus follows by \eqref{covzk1zk2asymp} that for any $1\le m\le l\le k$, the correlation coefficient of 
$(Z^{(m)}_n-\E[Z^{(m)}_n])/[m(np)^m]$ and $(Z^{(l)}_n-\E[Z^{(l)}_n])/[l(np)^l]$ tends to 1,
as $n\to\infty$. Then (ii) holds.

We next prove (iii).
If $np\rightarrow 0$ and $n^{2}p\rightarrow \infty$,
by \eqref{esm1nasym} and \eqref{varsmn} we now have that $\E[S_{m+1,n}]\sim {n^{m+1}p^m}/{m!}$ for $m\ge1$, and 
\begin{align*}
\V[S_{2,n}]\sim 2n^{2}p, \quad \V[S_{m+1,n}]\sim \frac{n^{m+1}p^m}{m!}, \quad m\ge2,
\end{align*}
which implies that for any $m\ge2$,
\begin{align*}
\E[S_{m+1,n}]=o(\E[S_{2,n}]), \quad \V[S_{m+1,n}]=o(\V[S_{2,n}]).
\end{align*}
Using Chebyshev's inequality, it thus follows from Proposition \ref{propstar} that for any $k\ge2$
\begin{align}\label{zknms2n}
\frac{(Z^{(k)}_n-\E [Z^{(k)}_n])-(S_{2,n}-\E[S_{2,n}])}{\sqrt{\V[S_{2,n}]}}
=\frac{\sum_{m=2}^km!\stirling{k}{m}(S_{m+1,n}-\E[S_{m+1,n}])}{\sqrt{\V[S_{2,n}]}}\inprob 0.
\end{align}
This says that the quantity $S_{2,n}$ makes a dominant contribution 
to the $k$-th order generalized Zagreb index for any $k\ge1$, if $np\rightarrow 0$ and $n^{2}p\rightarrow \infty$. 
On the other hand, applying Theorem 6.5 in \cite{janson2000rg} we have
\begin{align*}
\frac{S_{2,n}-\E[S_{2,n}]}{\sqrt{\V[S_{2,n}]}}\indist N(0,1),
\end{align*}
which, together with \eqref{zknms2n}, implies that (iii) holds.
\end{proof}

In particular, for $k=3$, an immediate consequence of Theorem \ref{thmzagreb} is that if $np\rightarrow c$ for some constant $c > 0$,
we have
\begin{equation*}
    \frac{1}{\sqrt{n}}\big(
        Z^{(1)}_n-c n,Z^{(2)}_n-(c^2+c) n, \cdots, Z^{(3)}_n-(c^3+3c^2+c) n
    \big)^{\top}\indist N(\bm{0},\bm{\Sigma}_3),
\end{equation*}
where the covariance matrix 
\begin{equation*}
\bm{\Sigma}_3=2c
\begin{pmatrix}
1 & 2c+1& 3c^2+6c+1\\
2c+1 & 4c^2+5c+1 & 6c^3+18c^2+11c+1\\
3c^2+6c+1 & 6c^3+18c^2+11c+1 & 9c^4+45c^3+63c^2+21c+1
\end{pmatrix}.
\end{equation*}

The next theorem, which is, in fact, a generalization of Theorem 3 in \cite{feng2013zagreb},
 establishes a necessary and sufficient condition for the asymptotic normality of the
generalized Zagreb indices of $G(n, p)$, as $n\rightarrow \infty$.

\begin{theorem}\label{nesucond}
For any given integer $k\ge 1$, let $Z^{(k)}_n$ be the $k$-th order generalized Zagreb index of a random graph $G(n, p)$. 
If $n^2p(1-p)\rightarrow\infty$, then
\begin{equation}\label{zknindist}
    \frac{Z^{(k)}_{n}-\sum_{m=1}^k\stirling{k}{m}\frac{n!}{(n-m-1)!}p^m}
    {\sqrt{n\sum\limits_{m=1}^k\sum\limits_{l=1}^k\stirling{k}{m}\stirling{k}{l}
\Big(ml(np)^{m+l-1}(1-p) 
+\sum\limits_{s=1}^{\min\{m,l\}}\frac{m!\,l!(np)^{m+l-s}(1-p^s)}{s!(m-s)!(l-s)!}\Big)}}
    \indist N(0,1).
\end{equation}
Conversely, if $(Z^{(k)}_n - a_n)/b_n \indist N(0, 1)$ for some constants $a_n$ and $b_n$, then $n^2p(1-p)\rightarrow\infty$.
\end{theorem}

{\bf Remark~} It is straightforward to see that the condition that $n^2p(1-p)\rightarrow\infty$ 
is equivalent to that $n^2p\to\infty$ and $n^2(1-p)\to \infty$,
while the quantities $n^2p/2$ and $n^2(1-p)/2$ are the asymptotic expected numbers of edges in $G(n,p)$ 
and the complement $\overline{G}(n, p)$, respectively. 
Then Theorem \ref{nesucond} states that for any $k\ge1$, the $k$-th order generalized Zagreb index of $G(n,p)$
has the asymptotic normality if $G(n,p)$ is not too close to either an empty graph or a complete graph. 
This also indicates that for the ultra-dense ER random graphs,
there remains a small gap in Theorem \ref{thmzagreb} that needs to be filled:
If the parameter $p$ satisfies that $n^2 (1-p)\to \infty$ but $n^2 (1-p)^3\to c$ for some constant $c\ge0$,
the single generalized Zagreb index $Z_n^{(m)}$ has asymptotic normality for any $1\le m\le k$, 
but it is unknown whether these random variables have the joint asymptotic normality. 

\begin{proof}[Proof of Theorem \ref{nesucond}] 
Here we employ a technique similar to that is used in \cite{feng2013zagreb}, and begin by proving the necessity part. 
If $n^2p(1-p)$ does not go to $\infty$,  then there exists a sequence $\{n_l, l = 1, 2,...\}$ 
such that $n_l^2p(1-p)\rightarrow c_1$ for some constant $c_1\in [0, \infty)$
with $p=p(n_l)\to 0$ or 1. Thus, we can conclude that $n_l^2p\to c_1$ or $n_l^2(1-p)\to c_1$. 
By Proposition \ref{propweakcon}, the subsequence $\{Z^{(k)}_{n_l} \}$ rules out asymptotic normality for any normalization,
and thus the necessity holds.

To prove the sufficiency part of Theorem \ref{nesucond},  we first note that \eqref{zknindist} is equivalent to 
\begin{equation}\label{zknan}
\frac{Z^{(k)}_{n}-\E[Z^{(k)}_{n}]}{\sqrt{\V[Z^{(k)}_{n}]}}\indist N(0,1),
\end{equation}
by \eqref{ezkasymp} and \eqref{vzkasymp}.
Also using Lemma \ref{lemjanson},
we next show that \eqref{zknan} holds under a stronger condition $n^2p(1-p)^3\rightarrow\infty$.
Similarly to the proof of Theorem \ref{thmstar}, it follows from Proposition \ref{propstar} that
\begin{align}\label{zknyna}
Z^{(k)}_{n}=\sum_{\alpha\in {\cal B}_n}(s_{n\alpha}-1)!\stirling{k}{s_{n\alpha}-1}I_{n\alpha}:=\sum_{\alpha\in {\cal B}_n}Y_{n\alpha},
\end{align}
where random variables $Y_{n\alpha}$ also have the same dependency graph $L_n$,
and $2\le s_{n\alpha}\le k+1$ for all $\alpha\in {\cal B}_n$.
To verify the condition of Lemma \ref{lemjanson} for $Z^{(k)}_{n}$ in the form of \eqref{zknyna}, 
by \eqref{ezkasymp} we can directly set 
\begin{align}\label{mn2}
M_n'':=2\sum_{m=1}^k \stirling{k}{m}n^{m+1}p^m=O\bigg(n^2p\Big(1+\sum_{m=1}^{k-1} (np)^m\Big)\bigg).
\end{align}
Suppose that two star graphs $\alpha_1,\alpha_2\in {\cal B}_n$ are given.  
Let $V_{12}$ be the vertex set of the induced subgraph $S_{n\alpha_1}\cup S_{n\alpha_1}$. 
Note that the cardinality of $V_{12}$ is at most $2(k+1)$, since both the sizes of $\alpha_1,\alpha_2$ are not greater than $k+1$.  
Consider the star graphs in ${\cal B}_n$ that share at least one edge with $\alpha_1$ or $\alpha_2$ and have exactly $s$ vertices out of $V_{12}$,
for $s=0, 1,\cdots,k-1$.  Recalling the constant $C_k$ defined for any $\alpha\in {\cal B}_n$ in the proof of (i) of Theorem \ref{thmstar}, 
here we can analogously define a constant $C_k'$ for any $\alpha_1,\alpha_2\in {\cal B}_n$.
Therefore, by \eqref{zknyna} we have 
\begin{align}\label{qn2}
    \sum_{\alpha\in \overline{N}_{L_n}}\E\big[|Y_{n\alpha}|\big|Y_{n\alpha_1},Y_{n\alpha_2}\big]
    &\le \max_{1\le s\le k}\Big(s!\stirling{k}{s}\Big)\cdot
    \sum_{\alpha\in \overline{N}_{L_n}}\E\big[I_{n\alpha}\big|I_{n\alpha_1},I_{n\alpha_2}\big]\notag\\  
    &\le \max_{1\le s\le k}\Big(s!\stirling{k}{s}\Big)\cdot C_k'\sum_{s=0}^{k-1}\binom{n}{s}p^s\notag\\
    &:=Q_n'',
\end{align}
which is of order $1+\sum_{s=1}^{k-1} (np)^s$. 
By \eqref{vzkasymp}, \eqref{mn2} and  \eqref{qn2}, we thus have
\begin{align*}
    \frac{M_n'' (Q_n'')^2}{(\V [Z^{(k)}_n])^{3/2}}&=O\bigg(\frac{n^2p(1+\sum_{s=1}^{k-1} (np)^s)^3}{\big[n^2 p(1-p)(1+\sum_{s=1}^{2(k-1)} (np)^s)\big]^{3/2}}\bigg)\\
    &=O\Big(\frac{1}{\sqrt{n^2p(1-p)^3}}\Big)\rightarrow 0.
\end{align*}
which implies that \eqref{zknan} holds, if $n^2p(1-p)^3\rightarrow\infty$.

%If $n^2(1-p)\rightarrow\infty$ and $np\rightarrow\infty$, we have $(Z^{(k)}_n-\E [Z^{(k)}_n])/(\V [Z^{(k)}_n])^{1/2}  \indist N(0,1)$ by (iii) of Theorem \ref{thmzagreb}.

Finally, we prove that \eqref{zknan} also holds under the desired condition 
$n^2p(1 - p)\rightarrow\infty$.
It is sufficient to consider the case where $n^2p(1-p)^3\rightarrow\infty$
does not hold. 
Consider any subsequence of positive integers $n_1<n_2<\cdots$
such that $n^2_l p(1-p)^3\to c_1$ for some constant $c_1\in[0,\infty)$,
as $n_l\to\infty$. 
Since $n_l^2 p(1-p)$ still tends to infinity, it follows that
\begin{equation*}
 \frac{n_l^2 p(1-p)^3}{n_l^2 p(1-p)}=(1-p)^2\rightarrow 0,
\end{equation*}
which implies that we must have $p\rightarrow 1$, and thus both $n_l p$ and $n_l^2 (1-p)$ tend to infinity.
Then, we now just need to prove that \eqref{zknan} holds under the condition that 
\begin{equation}\label{condnp}
np\to\infty,\quad n^2 (1-p)\to\infty,
\end{equation}
for simplicity omitting the subscripts $l$.
Under the condition \eqref{condnp}, it follows by \eqref{varsasymnpinf} that
\begin{align}\label{varsok} 
\V[S_{m+1,n}]=o\big(\V[S_{k+1,n}]\big), \quad m=1,2,\cdots,k-1, 
\end{align}
and further by Proposition \ref{propstar},  
\begin{align}\label{vtsims}
\V[Z^{(k)}_n]\sim (k!)^2\V[S_{k+1,n}].
\end{align}
Note that if $np\to\infty$, then $np^{m/(m+1)}\to \infty$ for any $m\ge1$.
Applying Theorem 6.5 in \cite{janson2000rg}, it follows that for any given integer number $m\ge1$,
\begin{align*}
\frac{S_{m+1,n}-\E[S_{m+1,n}]}{\sqrt{\V[S_{m+1,n}]}}\indist N(0,1).
\end{align*}
Therefore, we have that for any $k\ge1$,
\begin{align}\label{tknan}
\frac{Z^{(k)}_n-\E[Z^{(k)}_n]}{\sqrt{\V[Z^{(k)}_n]}}
&=\frac{\sum_{m=1}^{k} m!\stirling{k}{m} (S_{m+1,n}-\E[S_{m+1,n}])}{\sqrt{\V[Z^{(k)}_n]}}\notag\\
&=\frac{\sum_{m=1}^{k-1} m!\stirling{k}{m} (S_{m+1,n}-\E[S_{m+1,n}])}{\sqrt{\V\big[\sum_{m=1}^{k-1} m!\stirling{k}{m} S_{m+1,n}\big]}}\cdot
   \frac{\sqrt{\V\big[\sum_{m=1}^{k-1} m!\stirling{k}{m} S_{m+1,n}\big]}}{\sqrt{\V[Z^{(k)}_n]}}\notag\\
&\quad +\frac{k!(S_{k+1,n}-\E[S_{k+1,n}])}{k!\sqrt{\V[S_{k+1,n}]}}\cdot
  \frac{k!\sqrt{\V[S_{k+1,n}]}}{\sqrt{\V[Z^{(k)}_n]}}\notag\\
&\indist N(0,1),  
\end{align}
by \eqref{varsok}, \eqref{vtsims} and Slutsky's theorem.
This completes the proof of Theorem \ref{nesucond}.
\end{proof}

\section*{Acknowledgments}
We would like to thank the two anonymous referees for their valuable comments that have
improved the overall quality of the manuscript.

\end{document}